\documentclass[a4paper, 11 pt, twoside]{amsart}
\usepackage[left = 3.5cm, right = 3.5cm, headsep = 6mm,
footskip = 10mm, top = 35mm, bottom = 35mm, footnotesep=5mm, headheight =
2cm]{geometry}
\usepackage{graphicx}
\usepackage[expansion=false]{microtype}
\usepackage{booktabs, multirow,  tabularx}

\usepackage{MnSymbol}
\usepackage[usenames, dvipsnames]{xcolor}

\usepackage{tikz, tikz-cd}
\usepackage{enumitem}
\newlist{prooflist}{description}{1}
\setlist[prooflist]{font=\normalfont \itshape, labelindent = \parindent, leftmargin = 0pt}

\usepackage{verbatim}

\usepackage[unicode,bookmarks, pdftex, backref = page]{hyperref}
\hypersetup{colorlinks=true,citecolor=NavyBlue,linkcolor=NavyBlue,urlcolor=Orange, pdfpagemode=UseNone, breaklinks=true}

\newtheoremstyle{special-example}
  {}
  {}
  {}
  {\parindent}
  {\bfseries}
  {:}
  { }
  {}
  \theoremstyle{special-example}

\renewcommand{\tilde}{\widetilde}

\newcommand{\pp}{\mathbb P}

\newcommand{\Diff}{\mathrm{Diff}}

\newcommand{\I}{\mathrm{i}}

\newcommand{\HT}[1]{\lozenge_{#1}}

\newtheoremstyle{Lehn-it}
  {}
  {}
  {\itshape}
  {}
  {\bfseries}
  {$\;$\textmd{---}}
  { }
  {}

\newtheoremstyle{Lehn-up}
  {}
  {}
  {\upshape}
  {}
  {\bfseries}
  {$\;$\textmd{---}}
  { }
  {}

  \newtheoremstyle{up-list}
  {}
  {}
  {\upshape}
  {}
  {\bfseries}
  { }
  { }
  {}

\newtheoremstyle{Lehn-Bemerkung}
  {}
  {}
  {}
  {}
  {\itshape}
  {$\;$\textmd{---}}
  { }
  {}

\newtheoremstyle{citing}
  {}
  {}
  {\itshape}
  {}
  {\bfseries}
  {$\;$\textmd{---}}
  {.5em}
  {\thmnote{#3}}
  
\numberwithin{equation}{section}

{\theoremstyle{Lehn-it}
\newtheorem{theo}[equation]{Theorem}

\newtheorem{prop}[equation]{Proposition}

}
{\theoremstyle{Lehn-up}
\newtheorem{defin}[equation]{Definition}

}
{\theoremstyle{Lehn-Bemerkung}
\newtheorem{rem}[equation]{Remark}
}

{\theoremstyle{up-list}

}

{\theoremstyle{citing}
}

\DeclareFontFamily{OT1}{rsfs}{}
\DeclareFontShape{OT1}{rsfs}{n}{it}{<-> rsfs10}{}
\DeclareMathAlphabet{\curly}{OT1}{rsfs}{n}{it}

\DeclareMathOperator{\Proj}{\mathrm{Proj}}

\newcommand{\isom}{\cong}

\newcommand{\inverse}[1]{{#1}^{-1}}

\renewcommand{\epsilon}{\varepsilon}
\renewcommand{\phi}{\varphi}
\renewcommand{\theta}{\vartheta}

\newcommand{\kq}{{\mathcal Q}}

\newcommand{\IA}{{\mathbb A}}

\newcommand{\IC}{{\mathbb C}}

\newcommand{\IP}{{\mathbb P}}
\newcommand{\IQ}{{\mathbb Q}}

\newcommand{\IZ}{{\mathbb Z}}

\newcommand{\gothM}{{\mathfrak M}}
\newcommand{\gothN}{{\mathfrak N}}

\newcommand{\gothX}{{\mathfrak X}}

\title[Hodge structures and I-surfaces]{Degeneration of  Hodge structures on I-surfaces}

\author[S. Coughlan]{Stephen Coughlan}
\address{Stephen Coughlan\\Mathematisches Institut \\Lehrstuhl Mathematik VIII \\ Universit\"atsstra\ss e 30 \\ 95447 Bayreuth
\\Germany}
\email{stephen.coughlan@uni-bayreuth.de}

\author{Marco Franciosi}
\address{Marco Franciosi\\Dipartimento di Matematica\\Universit\`a di Pisa \\Largo B. Pontecorvo 5\\I-56127  Pisa\\Italy}
\email{marco.franciosi@unipi.it}

\author{Rita Pardini}
\address{Rita Pardini\\Dipartimento di Matematica\\Universit\`a di Pisa \\Largo B. Pontecorvo 5\\I-56127  Pisa\\Italy}
\email{rita.pardini@unipi.it}

\author{S\"onke Rollenske}
\address{S\"onke Rollenske\\FB 12/Mathematik und Informatik\\
Philipps-Universit\"at Marburg\\
Hans-Meerwein-Str. 6\\
35032 Marburg\\
Germany}
\email{rollenske@mathematik.uni-marburg.de}

\begin{document}
\maketitle
\begin{abstract}
Work of Green, Griffiths, Laza, and Robles suggests that the moduli space of (smoothable) stable surfaces should admit a natural stratification defined via Hodge theoretic data. In the case of stable surfaces with $K_X^2 = 1$ and $\chi(X) = 3$ we compute the Hodge type of all examples known to us and show that all predicted degenerations are geometrically realised.
\end{abstract}

\tableofcontents

\section{Introduction}

The study of surfaces of general type and their moduli spaces constructed by Gieseker is a classical topic of investigation.
 In recent years, additional interest has arisen from the construction of a modular compactification, the moduli space of stable surfaces.

Already the Gieseker moduli space $\gothM_{K^2,\chi}$ can be arbitrarily singular and, in general, the compactification $\overline{\gothM}_{K^2,\chi}$ also does not have good geometric properties. However, in their investigation of compactifications of period domains, Green, Griffiths, Laza, and Robles suggested the existence of a stratification of the  closure of the Giesecker moduli space  determined by Hodge-theoretic data.
In its coarsest form, the Hodge type reflects the weights with respect to the mixed Hodge structure of the cohomology classes contributing to $p_g$. Several new features (and difficulties) occur for $p_g\geq 2$.
We will explain some of this in Section \ref{Hodge stratification}, but want to mention here that the final aim would be to use even more information from the Hodge structure to study the geometry of the moduli space, compare \cite{GGR21}.

One of the easiest test cases for these ideas are minimal surfaces with $K^2_X = 1$ and $
\chi(X) = 3$, baptised I-surfaces by Green et al. These had been studied by the three last named authors in \cite{FPR17a} as part of their investigation of Gorenstein stable surfaces with $K_X^2 = 1$.
The interaction of the Hodge theoretic and geometric approaches, also manifest at the 2021 Meeting ``{\em Moduli and Hodge Theory, IMSA Miami}''
 has kept our interest in this class of surfaces alive and led to the papers
\cite{FPRR22} and \cite{CFPRR}.

In the present paper, our purpose is modest: we want to exhibit a zoo of examples of I-surfaces, mostly non-normal, and show how relatively easy geometric constructions
 realise all degenerations of the Hodge type predicted by general theory (cf. \cite{Robles16}). As our understanding of the moduli space of I-surfaces is still limited, there is no way we can claim to give a complete discussion in any sense.

\subsection*{Acknowledgements} We thank Mark Green, Phillip Griffiths, Radu Laza and Colleen Robles for showing us the role of Hodge theory
 in analysing the boundary components of the moduli space of stable surfaces,  and, above all,    for interesting and stimulating discussions on these topics.
 
 \hfill\break 
 M.F.  and R.P. are  partially supported by the project PRIN 
 2017SSNZAW$\_$004 ``Moduli Theory and Birational Classification"  of Italian MIUR and members of GNSAGA of INDAM.

\section{Stable surfaces and Koll\'ar's glueing}\label{Kollar-glueing}
The (semi-log)-canonical  reference for the basic definitions and properties of stable surfaces is \cite{KollarSMMP}.

We simply recall that a stable surface $X$ has ample $\IQ$-Cartier canonical divisor $K_X$ and semi-log-canonical singularities.

If  $X$ is a stable normal surface and
 $(X, x)$ is  an isolated surface singularity,  then $(X,x)$ is either  a rational or  elliptic singularity.
 In the second case,
considering the minimal  resolution $\sigma \colon  (\bar{Y}, E) \to (X,x)$, $x$ is said to be  simple  elliptic if $E$ is smooth of genus $1$,
 and is said to be a cusp if $p_a(E)=1$ and $E$ is a cycle   of rational curves.  The degree of the singularity is $-E^2$.


If   $X$ is a non-normal stable surface we consider its normalisation  $\pi\colon \bar X\to X$ and we recall that the non-normal locus $D\subset X$ and its preimage $\bar D\subset \bar X$ are pure of codimension $1$, i.\,e., they are  curves. Since $X$ has ordinary double points at the generic points of $D$,  the map on normalisations $\bar D^\nu\to D^\nu$ is the quotient by an involution $\tau$.
Koll\'ar's glueing principle says that $X$ can be uniquely reconstructed from $(\bar X, \bar D, \tau\colon \bar D^\nu\to \bar D^\nu)$ via the following two push-out squares:
\begin{equation}\label{diagr: pushout}
\begin{tikzcd}
    \bar X \dar{\pi}\rar[hookleftarrow]{\bar\iota} & \bar D\dar{\pi} & \bar D^\nu \lar[swap]{\bar\nu}\dar{/\tau}
    \\
X\rar[hookleftarrow]{\iota} &D &D^\nu\lar[swap]{\nu}
    \end{tikzcd}
\end{equation}
More precisely, we have the following.

\begin{theo}[{\cite[Thm.~5.13]{KollarSMMP}}]\label{thm: triple}
Associating to a stable surface $X$ the triple $(\bar X, \bar D, \tau\colon \bar D^\nu\to \bar D^\nu)$ induces a one-to-one correspondence
 \[
  \left\{ \text{\begin{minipage}{.12\textwidth}
 \begin{center}
         stable  surfaces
 \end{center}
         \end{minipage}}
 \right\} \leftrightarrow
 \left\{ (\bar X, \bar D, \tau)\left|\,\text{\begin{minipage}{.37\textwidth}
   $(\bar X, \bar D)$ log-canonical pair with
  $K_{\bar X}+\bar D$ ample, \\
   $\tau\colon \bar D^\nu\to \bar D^\nu$  involution s.th.\
    $\Diff_{\bar D^\nu}(0)$ is $\tau$-invariant.
            \end{minipage}}\right.
 \right\}.
 \]
 where $\Diff_{\bar D^\nu}(0)$ is the different (see \cite[5.11]{KollarSMMP} for the definition).
\end{theo}
%
%
%
%
%

\subsection{I-surfaces}\label{ssection: I-surf}

An \emph{I-surface}  $X$  is a stable surface with   $K_X^2=1$ and $\chi(X)=3$. If $X$ is  smooth this is equivalent to saying
 that  $X$ is a minimal surface of general type with $K_X^2=1$, $p_g(X)=2$ and $q(X)=0$.

By Theorem \ref{thm: triple}  a triple $(\bar X, \bar D, \tau)$ corresponds to an I-surface  if and only if the following four conditions are satisfied:

\begin{description}
\item[Stable pair condition]  $(\bar X, \bar D)$ is an lc pair such that $K_{\bar X}+\bar D$ is an ample $\IQ$-Cartier divisor.
\item[$K_X^2$-condition] $(K_{\bar X}+\bar D)^2=1$.
 \item[Glueing condition] $\tau\colon \bar D^\nu\to \bar D^\nu$ is an involution such that $\Diff_{\bar D^\nu}(0)$ is $\tau$-invariant.
 \item[$\chi$-condition]  The holomorphic Euler-characteristic of the non-normal locus $D$ is  $\chi(D) = 3-\chi(\bar X)+\chi(\bar D)$.
\end{description}

In \cite{FPR17a} it was shown  that the classical description of smooth surfaces of general type with
$K_X^2=1$ and $\chi(X)=3$  extends to  the Gorenstein case, i.e.:

\begin{itemize}
\item  a Gorenstein I-surface $X$  is canonically embedded as a hypersurface of degree $10$ in (the smooth locus of) $\IP(1,1,2,5)$;

 \item the moduli space
$\overline\gothM_{1,3}^{(Gor)}$ of Gorenstein stable surfaces with $K^2=1$ and $\chi=3$ is irreducible and rational of dimension $28$.
\item  for a Gorenstein I-surface $X$, the bicanonical map is a degree 2 morphism $\phi_2\colon X\to \kq_2 \subset \IP^3$, where $\kq_2$  is  the quadric cone, branched on the vertex $o$   and on a quintic section $B$ of $\kq_2$ not containing $o$.
\item  conversely, if $B$ is a quintic section of $\kq_2$ such that  $(\kq_2, \frac 12 B)$ is a log-canonical pair,
then the double cover of $\kq_2$ branched on $B$ and  $o$ is a stable  I-surface;  it is  Gorenstein if $B$ does  not contain $o$, else it is 2-Gorenstein.
\end{itemize}

\section{Stratification of $\bar \gothM$ from Hodge structures}\label{Hodge stratification}

\subsection{The mixed Hodge structure on $H^2(X)$}\label{subsection:hodge}
Here we follow roughly the exposition of \cite{anthes20}.

The stratification we will define and study later is motivated by work (partially in progress) of Green, Griffiths, Kerr, Laza and Robles \cite{GGR21, KPR,Robles16a,Robles16} about degenerations of Hodge structures.
Roughly, there should be a stratification of the moduli space of the surfaces under investigation, according to the type of polarised mixed Hodge structure on $H^2(X)$.
It should be noted that the details about this Hodge-theoretic stratification are not fully understood yet. 
Therefore, we can only give an informal description.
We refer to Robles' exposition \cite{Robles16} and the references therein for more details; for the basic theory of mixed Hodge structures see Durfee's short introduction \cite{Durfee:1983} and the comprehensive account by Peters and Steenbrink \cite{Peters-Steenbrink}.


Given a flat degeneration   $\gothX\to\Delta$, where $\Delta\subset \IC$ is the unit disc and where all fibres $\gothX_{s}$, $s\in\Delta^* = \Delta-0$
 are smooth  and projective, we can can associate a \emph{limiting polarised mixed Hodge structure} (LMHS for short) with the family of Hodge structures $H^2(\gothX_{s};\IC)$
(see \cite{GGR21}). Moreover,  we can compare this LMHS  with 
 Deligne's natural mixed Hodge structure on the cohomology of the special fibre $H^n(\gothX_{0};\IC)$  via the specialisation map.
In general, the  two mixed Hodge structures are related by the Clemens--Schmid exact sequence (cf. \cite{Cle77}) and its generalizations (cf. \cite{KL19, KL20}).

For our application we consider the case where $X= \gothX_{0}$ is a stable surface. In this case
 we consider  Deligne's mixed Hodge structure on $X$, since
$X$ is   \emph{cohomologically  insignificant}. That is,
 for every $n$, the Deligne mixed Hodge structure on $H^n(\gothX_0;\IC)$ and the limiting mixed Hodge structure on $H^{n}(\gothX_s;\IC)$ agree on $(p,q)$-components where $pq = 0$.
  This  follows since
a projective variety with at worst du Bois singularities is cohomologically insignificant by
a result of Steenbrink \cite[Theorem~2]{Steenbrink:1980}, and
 semi-log-canonical surfaces are Du Bois  (cf.\ Koll\'{a}r \cite[Corollary~6.32]{KollarSMMP} or Kov\'{a}cs, Schwede, Smith \cite[Theorem~4.16]{kss10}).
See also  \cite[Thm 9.3 and Cor. 9.9]{KL19} for a generalisation of the Clemens--Schmid exact sequence and applications to slc varieties.

Now we focus on \emph{I-surfaces}  (see \S \ref{ssection: I-surf}).  Assume that $X$ is smoothable.
Since $h^{2,0}(X) = h^{0,2}(X) = p_{g}(X) = 2$ for a smooth minimal surface $X$ of general type satisfying $K_{X}^2 = 1$ and $\chi(X) = 3$, our case of interest corresponds to the Hodge numbers $h = (2,h^{1,1},2)$.
That is, the Deligne splitting $H^2(X;\IC) = \bigoplus_{p,q} H^2(X;\IC)^{(p,q)}$ 
 of the limiting mixed Hodge structure  has a Hodge diamond (indicating $\dim H^2(X;\IC)^{(p,q)} $) of the form
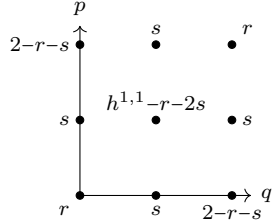
\begin{center}
\begin{tikzpicture}[every circle/.style={radius = 0.05}]
\draw [->] (0,0) -- (2.25,0);
\draw [<-] (0,2.25) -- (0,0);
\node [above] at (0,2.25) {$\scriptstyle{p}$};
\node [right] at (2.25,0) {$\scriptstyle{q}$};

\draw [fill] (0,2) circle;
\draw [fill] (0,1) circle;
\draw [fill] (0,0) circle;
\draw [fill] (1,2) circle;
\draw [fill] (1,1) circle;
\draw [fill] (1,0) circle;
\draw [fill] (2,2) circle;
\draw [fill] (2,1) circle;
\draw [fill] (2,0) circle;

\node [below left] at (0,0) {$\scriptstyle{r}$};
\node [above right] at (2,2) {$\scriptstyle{r}$};
\node [left] at (0,1) {$\scriptstyle{s}$};
\node [below] at (1,0) {$\scriptstyle{s}$};
\node [right] at (2,1) {$\scriptstyle{s}$};
\node [above] at (1,2) {$\scriptstyle{s}$};
\node [left] at (0,2) {$\scriptstyle{2-r-s}$};
\node [below] at (2,0) {$\scriptstyle{2-r-s}$};
\node [above] at (1,1) {$\scriptstyle{h^{1,1}-r-2s}$};
\end{tikzpicture}
\end{center}
where $r,s\geq 0$, $r+s\leq 2$ and $r+2s\leq h^{1,1}$.

\begin{defin}\label{def:Hodge type}
Let $X$ be an I-surface (see \S \ref{ssection: I-surf}).
Then $X$ is said to be of \emph{Hodge type $\HT{r,s}$} if $r = \dim H^2(X;\IC)^{(0,0)}$ and $s = \dim H^2(X;\IC)^{(1,0)} $, where $H^2(X;\IC)^{(p,q)}$ is the $(p,q)$-component of Deligne's mixed Hodge structure on $H^2(X;\IC)$.
\end{defin}


Among these Hodge types there is the so-called  {\em polarised relation},  defined by  $\HT{r,s}\leq \HT{t,u}$ if and only if $r\leq t$ and $r+s\leq t+u$.
This relation comes from representation theory and has a geometric meaning which follows  by the analysis  of the extension of period domain  and the period map
(cf.\ Robles \cite{Robles16}).
Roughly speaking $\HT{r,s}\leq \HT{t,u}$ means that the Hodge type $ \HT{t,u}$ is  ``more degenerate'' than  $\HT{r,s}$.

In our case  the polarised  relations among Hodge types for an I-surface are
illustrated in the Degeneration Diagram in Figure \ref{fig:hodge-diagram-diagram} (cf.  \cite[Example 4.10]{Robles16} for a detailed explanation).

 Smooth varieties have Hodge structures of pure weight, that is, they are of type $\HT{0,0}$.
Our purpose is to show that we can realise in several different ways  all the polarised relations illustrated in Figure  \ref{fig:hodge-diagram-diagram}
 by  considering the stratification   given by the singularities of  I-surfaces.
We will ignore $h^{1,1}$ just as  we will ignore canonical surface singularities.

\begin{figure}\centering
\begin{tikzpicture}
 \node (00)
 {$\HT{0,0}$};
 \node[right of=00,node distance=1.5cm] (01)
 {$\HT{0,1}$};
 \node[below right of=01,node distance=1.5cm] (02)
 {$\HT{0,2}$};
 \node[above right of=01,node distance=1.5cm] (10)
 {$\HT{1,0}$};
 \node[below right of=10,node distance=1.5cm] (11)
 {$\HT{1,1}$};
 \node[right of=11,node distance=1.5cm] (20)
{$\HT{2,0}$};

 \path (00) edge (01);
 \path (01) edge (10)
 			edge (02);
 \path (10) edge (11);
 \path (02) edge (11);
 \path (11) edge (20);
\end{tikzpicture}
\caption{Degeneration diagram for the Hodge types.}\label{fig:hodge-diagram-diagram}
\end{figure}
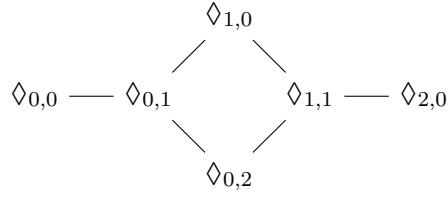

\subsection{The Mayer--Vietoris exact sequence}

Let us provide some information on how to compute Deligne's mixed Hodge structure on the second cohomology  and more specifically the Hodge type in practice.

\begin{prop}
There are Mayer--Vietoris type exact sequences of mixed Hodge structure in the following situations:
\begin{enumerate}
  \item  Let $ X$ be a demi-normal surface,   $\pi\colon \bar X\to X$ its normalisation,  $D\subset X$ the non-normal locus and $\bar D\subset \bar X$ its preimage  (see \S \ref{Kollar-glueing}).
  Then the  commutative diagram associate to the normalisation map
  \begin{equation*}
  \begin{tikzcd}
    \bar X \dar{\pi}\rar[hookleftarrow]{\bar\iota} & \bar D\dar{\pi|_{\bar D}}
    \\
    X\rar[hookleftarrow]{\iota} &D
  \end{tikzcd}
\end{equation*}
induces an exact sequences of mixed Hodge structures
\begin{equation}\label{eq: MHS normalisation}
\begin{tikzcd}
&H^1(X, \IC) \rar { (\pi^* , \iota^*)} & H^1(\bar X, \IC) \oplus H^1(D, \IC) \rar{ \bar\iota^*- \pi|_D^*}&
H^1(\bar D, \IC) \rar &{}\\
{}\rar&H^2(X, \IC) \rar { (\pi^* , \iota^*)} & H^2(\bar X, \IC) \oplus H^2(D, \IC) \rar{ \bar\iota^*- \pi|_D^*}& H^2(\bar D, \IC) \rar & \dots
\end{tikzcd}
\end{equation}
\item Let $\bar X$ be a normal surface and $\sigma \colon \bar Y \to \bar X$ a resolution. Let $S = \{p_1, \dots, p_k\}$ be the singular points of $X$ and $E_i$ the exceptional curve over $p_i$. Then there is an exact sequence
\begin{equation}\label{eq: MHS resolution}
\begin{tikzcd}
&H^1(\bar X, \IC) \rar { \sigma^*} & H^1(\bar Y, \IC)  \rar &
\bigoplus_i H^1(E_i, \IC) \rar & {}\\
{}\rar&H^2(\bar X, \IC) \rar { \sigma^*} & H^2(\bar Y, \IC) \rar &  \bigoplus_i H^2(E_i, \IC) \rar & {}
\end{tikzcd}
\end{equation}
\item
Let $C$ be a curve with singular points $p_1, \dots, p_r$ and $\nu \colon C^\nu = \bigsqcup_i C^\nu_i \to C$ the normalisation.
Then there is an exact sequence
\begin{equation}\label{eq: MHS normalisation curve}
\begin{tikzcd}
0 \rar &H^0(C, \IC) \rar  & H^0(C^\nu, \IC)\oplus \bigoplus_i H^0(p_i, \IC)   \rar &
\bigoplus_i H^0(\inverse\nu(p_i), \IC) \rar &{}\\
{}\rar&H^1(C, \IC) \rar { \nu^*} & \bigoplus_i H^1(C^\nu_i, \IC) \rar &  0
\end{tikzcd}
\end{equation}

\end{enumerate}

\end{prop}

\begin{proof} The Proposition
 follows from  \cite[Corollary-Definition 5.37]{Peters-Steenbrink}. Concerning  {\em (iii)},  a detailed exposition in the case where  $C$ has only nodes  is given in \cite[\S 4]{Durfee:1983}.  
\end{proof}

In particular for a stable surface $X$,  if $(X,p)$  is  a simple elliptic singularity with
 minimal  resolution $\sigma \colon  (\bar{Y}, E) \to (X,p)$,
 then $E$ is a smooth  elliptic curve, hence $H^1(E, \IC)$ carries a pure Hodge structure of weight $1$. If
$p$ is a cusp then $E$ is a cycle of rational curves, hence $H^1(E, \IC)$ carries a mixed Hodge structure concentrated in weight $0$.
 The restriction maps in the  Mayer--Vietoris exact sequences described above  can be used  to compute the mixed Hodge structure on $X$
 (cf.  \cite[\S 2]{anthes20}).

\begin{rem}\label{rem: ideas}
We collect some slogans about what determines the mixed  Hodge type on the second cohomology, that follow from the above exact sequences:
\begin{itemize}
\item Smooth varieties have Hodge structures of pure weight.
 \item  Rational singularities do not affect the Hodge type of a stable surface. In particular,  a surface $X$ with only rational singularities has Hodge structure of pure weight.
 \item Isolated irrational  singularities may introduce the HS of the exceptional curve.
 \item Non-normal glueings may introduce the HS of the non-normal locus.
 \item On a curve, the weight filtration on the first cohomology captures the genus of the normalisation.
\end{itemize}
\end{rem}

\begin{rem}
 We want to emphasise again that the Hodge type $\HT{r,s}$ is a numerical datum, which forgets a lot of information contained in the mixed Hodge structure.
  In particular, both the integral structure as well as the extension data of the graded pieces should capture much more of the geometry of the surface, including information about the rational singularities.
   We refer to \cite[\S 5.2]{GGR21} for the period mapping and a detailed description of the extension data for I-surfaces.
\end{rem}

\section{Examples of degenerations in  I-surfaces}

\subsection{Double covers of the cone}\label{sect: cone 1}

Smooth I-surfaces can be realised as double covers of the quadric cone $\kq_2 \subset \IP^3$, branched over  on the vertex $o$   and on a smooth quintic section $B$ of $\kq_2$ not containing $o.$
Therefore we can obtain degenerations of a smooth I-surface by letting $B$ acquire suitable singularities (see \S \ref{ssection: I-surf}).

 Since   the quadric cone can be seen as the image of the second Veronese map of $ \IP(1,1,2)$,
we  identify $ \kq_2$ with  $ \IP(1,1,2)$  and
we   describe  I-surfaces that arise as double covers of the quadric cone in $\IP^3$  as a hypersurfaces $X$  of degree ten in $\IP(1,1,2,5)$ given by an equation
\[z^ 2 - f_{10}(x_0, x_1, y) = 0.  \]

Denote by $B = \{f_{10} = 0 \} \subset \IP(1,1,2)$ the branch locus of the double cover.
Then the following hold:
\begin{itemize}
 \item $X$ has slc singularities if and only if $(\IP(1,1,2), \frac 12 B)$ is an lc pair.
 \item $X$ is non-normal if and only if $f$ has a multiple factor.
 \item $X$ is Gorenstein if and only if the branch curve $B$ does not pass through the vertex, that is, the singular point  $(0:0:1)$ of $\IP(1,1,2)$.
 \end{itemize}

 Let us quickly review what kind of singularities we can allow on the branch curve $B $ in a smooth point of $\IP(1,1,2)$ so that $X$ is slc.

 If  $p\in B$  is an isolated singularity of $B$, then one of the following occurs:

\begin{itemize}
\item $p$ is a {\em negligible} singularity:  $p$ is a double point or a triple point such that every point infinitely near to $p$  is at most double for the strict transform of $B$.
The preimage of $P$ in $X$ is a canonical singularity.

\item $p$ is a quadruple point such that for every point $q$ infinitely near to $p$  the local intersection number at $q$  of the strict transform of $B$ and the exceptional curve is at most 2.
The preimage of $p$ in $X$ is an elliptic Gorenstein singularity of degree $2$.
The exceptional curve in the minimal resolution is a smooth elliptic curve if and only if the quadruple point is ordinary, else it is a cycle of rational curves.
\item $p$ is a $[3,3]$-point, namely $p$ is a triple point with an infinitely near triple point $p_1$ such that  for every point $q$ infinitely near to $p_1$  the local intersection number at $q$
of the strict transform of $B$ and the exceptional curve is at most 2. The preimage of $p$ in $X$ is an elliptic Gorenstein singularity of degree $1$.
The exceptional curve in the minimal resolution is a smooth elliptic curve if and only if the infinitely near triple point is ordinary, else it is a cycle of rational curves.
\end{itemize}

\subsubsection{Normal double covers of the cone}

Clearly,  a smooth I-surface has Hodge type $\HT{0,0}$ and we want to show quickly that all possible degenerations  of the  Hodge type already occur for normal surfaces. A different stratification by number and degree of elliptic points was given in \cite{FPR17a}. But since  cusps were not considered in loc.\ cit., we cannot  cite all desired degenerations directly from there. For the reader's convenience, we reproduce the resulting stratification together with the corresponding Hodge type in Figure \ref{fig: old stratification}. Some information on the general element in each stratum is given in Table \ref{tab: normal strata}, where
\[ \gothN_{d_1, \dots, d_k} = \left\{X\in \overline\gothM_{1,3}^{(Gor)}
\left|\,\text{\begin{minipage}{.45\textwidth}
   $X$ is normal and has exactly $k$ elliptic singularites of degree $d_1\leq \dots\leq  d_k$
            \end{minipage}}\right.
 \right\}.
\]

\begin{figure}
 \caption{Hodge type for the normal surfaces from  \cite{FPR17a}}\label{fig: old stratification}
\[
\begin{tikzcd}[arrows=dash, column sep = small]
 {}&&& \gothM_{1,3}
 \arrow{dl}\arrow{dr}
 &&&& \HT{0,0}\dar[->]
\\
 && \gothN_1
 \arrow{dl}
 \dar\arrow{dr}
 & & \gothN_2\dar\arrow{dl} & & & \HT{0,1}\dar[->]
\\
 &\gothN_{1,1}^R & \gothN^E_{1,1}\dar\arrow{dr} & \gothN_{1,2}\dar & \gothN_{2,2}&&
  & \HT{0,2}\\
&&\gothN_{1,1,1} & \gothN_{1,1,2}
 \end{tikzcd}
\]
\end{figure}
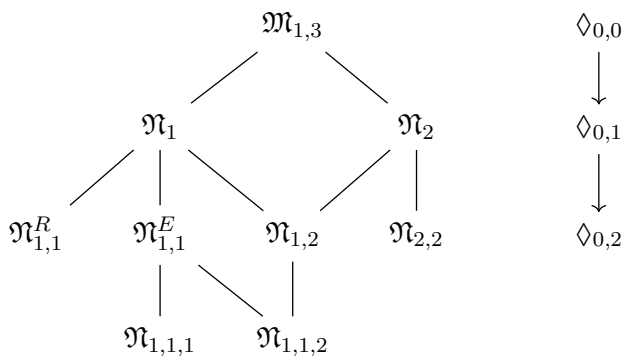

\begin{table}[h]\caption{Irreducible  strata of normal surfaces in $\overline\gothM_{1,3}^{(Gor)}$}\label{tab: normal strata}
 \begin{center}
 \begin{tabular}{cclc}
 \toprule
stratum & dimension & minimal resolution $\tilde X$ & $\kappa(\tilde X)$\\
\midrule
$\gothN_\varnothing=\gothM_{1,3}$& $28$ & general type & $2$\\
$\gothN_2$ & $20$  & blow up of a K3-surface & $0$\\
$\gothN_1$ & $19$  & minimal elliptic surface with $\chi(\tilde X)=2$ & $1$\\
$\gothN_{2,2}$ & $12$  & rational surface& $-\infty$\\
$\gothN_{1,2}$ & $11$ & rational surface& $-\infty$\\
$\gothN_{1,1}^R$ & $10$ &rational surface& $-\infty$\\
$\gothN_{1,1}^E$ & $10$ &blow up of an Enriques surface & $0$\\
$\gothN_{1,1,2}$ & $2$ & ruled surface with $\chi(\tilde X)=0$& $-\infty$\\
$\gothN_{1,1,1}$ & $1$  & ruled surface with $\chi(\tilde X)=0$& $-\infty$\\
\bottomrule
\end{tabular}
 \end{center}
\end{table}

Now we explain one possibility to realise everything in Figure
\ref{fig:hodge-diagram-diagram}: fix two general points $p_1, p_2$ on quadric the cone and take three general planes $H_3, H_4, H_5$ containing both points. Now consider the family of all pairs of planes $H_1, H_2$ such that if one takes $
B =
\sum_i H_i$ as branch divisor then the resulting surface is a normal stable surface $X$.

Then the following happens:
\begin{enumerate}
 \item if $H_1, H_2$ are general, then $X$ has canonical singularities;
 \item if $H_1$ is general and $H_2$ passes through $p_2$ in a general way, then  $p_2$ is an ordinary  quadruple point for $B$ and $X$ has one simple elliptic singularity of degree $2$;
 \item if $H_1$ is general and $H_2$ is tangent to $H_3$ in  $p_2$, then $p_2$ is a degenerate quadruple point and  $X$ has one  cusp singularity of degree $2$;
   \item if $H_i$ passes through $p_i$ in a general direction, then  then  $p_1$ and $p_2$ are ordinary quadruple points and $X$ has two elliptic singularities of degree $2$;
  \item if $H_1$ passes through $p_1$ in a general direction and  $H_2$ is tangent to $H_3$ in  $p_2$, then $X$ has one  elliptic and one cusp singularity of degree $2$;
   \item if $H_i$  is tangent to $H_3$ at $p_i$, then $X$ has two cusp singularities of degree $2$.
\end{enumerate}
According to general principles (\eqref{eq: MHS resolution}), this should realise all degenerations in Figure \ref{fig:hodge-diagram-diagram}, if we check that the exceptional  curves in the resolution give rise to independent classes in $H^2(X)$.
This can be easily checked using the classification of cases given in \cite{FPR17a}, that we have summarized in Table \ref{tab: normal strata}.

 In case ({\em ii}) the  minimal resolution  $\bar Y$ of $X$ is the blow-up of a K3 surface. Denoting by $E_2$ the elliptic curve coming from the resolution of the singular point over $P_2$, equation  \eqref{eq: MHS resolution} reads as follows
\begin{equation}\nonumber
\begin{tikzcd}
 0= H^1(\bar Y, \IC)  \rar &
 H^1(E_2, \IC)
{}\rar&H^2( X, \IC) \rar { \sigma^*} & H^2(\bar Y, \IC)
\end{tikzcd}
\end{equation}

In this case $H^1(E_2, \IC)$  carries a pure Hodge structure of weight 1 with $h^{1,0}(E_2) = 1.$
Therefore the mixed Hodge structure of $X$ is of type  $ \HT{0,1}$, that is,  we have  realised  the degeneration of Hodge types $\HT{0,0}\to \HT{0,1}.$

If we are in case ({\em iv})  then our surface $X$ has two simple elliptic singularities and the minimal resolution  $\bar Y$ of $X$  is a rational surface.
Arguing as above, equation  \eqref{eq: MHS resolution} shows that the classes of the exceptional divisors $E_1$ and $E_2$ give independent classes in $H^2(X)$, hence we obtain a  further degeneration of Hodge types  $\HT{0,1} \to \HT{0,2}$.

If we are in case ({\em iii})  then  $X$ has one  cusp singularity of degree $2$ over $P_2$ and its minimal resolution is again the blow-up of a K3 surface.
This is in fact  a degeneration of  case ({\em ii}).
Denoting  by $E_2= \sum_{i =1}^rE_{2,i}$ the cycle of rational curve coming from the resolution of the singular point over $p_2$,
and by  $\{q_1, \ldots, q_r\}$ the set of singular points of $E_2$,
  equation  \eqref{eq: MHS resolution}  and equation \eqref{eq: MHS normalisation curve} show that we get a mixed Hodge structure of type $ \HT{1,0}$. Indeed we get
\begin{equation}\nonumber
\begin{tikzcd}
 0= H^1(\bar Y, \IC)  \rar &
 H^1(E_2, \IC)
{}\rar&H^2( X, \IC) \rar { \sigma^*} & H^2(\bar Y, \IC) \rar &  0
\end{tikzcd}
\end{equation}
To study $H^1(E_2, \IC)$ we consider the exact sequence \eqref{eq: MHS normalisation curve}
\begin{equation}\nonumber
\begin{tikzcd}
0 \rar & \IC=H^0(E_2, \IC) \rar  &  \bigoplus_i H^0(E_{2,i}, \IC)\oplus \bigoplus_i H^0(q_i, \IC)  \rar &{}  \\
\rar &  \bigoplus_i H^0(\inverse\nu(q_i), \IC)  \rar &
 H^1(E_2, \IC) \rar { \nu^*} & \bigoplus_i H^1(E_{2,i}, \IC)
\end{tikzcd}
\end{equation}
Now it is
$ \bigoplus_i H^0(E_{2,i}, \IC) \cong   \IC^r \cong  \bigoplus_i H^0(q_i, \IC) $,
 $\bigoplus_i H^0(\inverse\nu(q_i), \IC) \cong \IC^{2r}$ and
$ \bigoplus_i H^1(E_{2,i}, \IC) =0 $. Moreover $(H^1(\bar Y))^{0,0}= (H^2(\bar Y))^{0,0}=0$, hence  $(H^2(X))^{0,0}$   is isomorphic to the $(0,0)$ components of $H^1(E_2)$. Summing up,
we have obtained  the degeneration  $\HT{0,0}\to \HT{0,1} \to \HT{1,0}$.

In case ({\em v})  and case ({\em vi}) we repeat the above arguments  for the two singular points, obtaining  the degeneration $\HT{0,0}\to \HT{0,1} \to \HT{1,0} \to \HT{1,1}$ and
$\HT{0,0}\to \HT{0,1} \to \HT{1,0} \to \HT{1,1} \to \HT{2,0}$, respectively.

\begin{rem}
In Table \ref{tab: normal strata} we find  degenerations $X$ with three  elliptic points and minimal resolution  a (possibly non-minimal)  surface $\tilde X$ ruled over an elliptic curve. Therefore, the only possible exceptional elliptic cycles  are smooth elliptic cuves embedded as multi-sections in $\tilde X$. Thus the relevant part of  \eqref{eq: MHS resolution} becomes
\[
 \begin{tikzcd}
0 \rar & H^1(\tilde X, \IC)  \rar[hookrightarrow] &
\bigoplus_{i=1}^3 H^1(E_i, \IC) \rar&H^2(X, \IC) \rar & H^2(\tilde X, \IC)
\end{tikzcd},
\]
because each multisection carries, via its map to the base curve of the ruled surface the first rational cohomology. Thus the Hodge type turns out to be $\HT{0,2}$. Note that none of the elliptic points can degenerate to a cusp because of the geometry of the minimal resolution. If we want to degenerate the Hodge type further we need to degenerate $X$ to a non-normal surface.

\end{rem}
\subsubsection{Non-normal double covers of the cone}

 If we want to consider non-normal double covers of the cone then we have a non-reduced branch divisor, that is,
 \[ f = gd^2 \text{ and }  B = B_0 + 2 D,\]
 and the normalisation diagram \eqref{diagr: pushout} becomes
 \[
  \begin{tikzcd}
    &\bar X \dar{\pi}\rar[hookleftarrow]{\bar\iota} \arrow{ddl}[swap]{\text{branched over $B_0$ and vertex}} & \bar D\dar{\text{branched over $B_0\cap D$ and vertex}}[swap] {\pi}
    \\
    &X\rar[hookleftarrow]{\iota}\arrow{dl}
    &D
\\
    \IP(1,1,2)
  \end{tikzcd}
  \]

If both $B_0$ and $D$ are general then the normalisation is smooth and so are the conductor curves. The possibilities are listed in Table \ref{tab: branch general}.
The first two cases yield Gorenstein non-normal I-surfaces and are described in \cite{FPR17a}. The other three cases are new and  yield   2-Gorenstein I-surfaces,
which have singular normalisation.

The list is obtained by considering the possible degrees of $B_0$ and  $D$. The genus of $D$ follows 
 by using the adjunction formula for a complete intersection curve in $\pp^3$.    $K_X$ and $K_{\bar D}$ can be computed
via the Riemann-Hurwitz formula
looking at the double covers  $\bar X \to \IP(1,1,2)$ branched over $B_0$ and the vertex,  respectively,  $\bar D \to D$ branched over $B_0\cap D$ and the vertex.

Note that a smooth  curve in $\kq_2$
passes through the vertex if its degree is odd. 
If such a curve is a component of $B$ then $X$ is not Gorenstein. 

The curve $D$ can have at most nodes as singularities. If $p$ is a node of $D$ then  $p\not \in B_0$ and
in the double cover we get a degenerate cusp locally analytically given by the equation $z^2 - u^2v^2=0$.
If $p$ is a smooth point of $D$, then one of the following occurs:
\begin{itemize}
 \item $p\not \in B_0$ and the preimage in $X$ is a normal crossing point.
 \item $D$ and $B_0$ intersect transversely at $p$ and the preimage in $X$ is a pinch point.
 \item The local intersection multiplicity $(D, B_0)_p = 2$ and $B_0$ is either smooth or has an $A_n$ singularity at $p$.
 The preimage is a degenerate cusp, usually denoted with $T_{2,q,\infty}$ for $q\geq 3$ and locally described by the equation $z^2- u^2(v^2-u^{q-2})$.

 One could view $T_{2,3,\infty}$ as a non-normal version of the elliptic or cusp singularity coming from a $[3,3]$ point, while the others are non-normal versions of the quadruple point.

 \end{itemize}
This can be checked by blowing up and analysing the cases  similarly to what is done in \cite{liu-rollenske12a} or from the point of view of log-canonical threshold  in \cite[6.5]{Corti-Kollar-Smith}.

\begin{table}
\caption{Non-normal double covers of $\IP(1,1,2)$  for general $D$ and $B_0$} \label{tab: branch general}
 \begin{tabular}{ccclcc}
  \toprule
  Hodge-Type & $\deg d$ & $ \deg g$ & $ \bar X$ & $g(\bar D)$ & $g(D)$\\
  \midrule
  $\HT{0,2}$& 2 & 6& del Pezzo of degree 1 & 2 & 0\\
  $\HT{0,2}$& 4 & 2& $\IP^2$ & 3 & 1\\
  \midrule
 $\HT{0,1}$ & 1 & 8&  K3 & 1 & 0\\
  $\HT{0,2}$& 3 & 4&  del Pezzo of degree 4 & 2 & 0\\
  $\HT{0,2}$ & 5 & 0& two copies of $\IP(1,1,2)$ & 2 + 2 & 2\\
  \bottomrule
 \end{tabular}
\end{table}

Equipped with this knowledge of possible singularities we will use the ideas from Remark \ref{rem: ideas}
to show that in each of the cases in Table
\ref{tab: branch general}, an appropriately chosen degeneration of the branch data will allow for all possible degenerations of the Hodge type depicted in Figure \ref{fig:hodge-diagram-diagram}.

\begin{description}
 \item[$(\deg d, \deg g) = (2,6)$] In this case \eqref{eq: MHS normalisation} gives an exact sequence \[0 \to H^1(\bar D, \IC) \to H^2 (X, \IC) \to H^2(\bar X, \IC) \to 0\]
  where $\bar X$ is a del Pezzo surface of degree one,  $D$ a rational curve and $\bar D$ is a genus 2 curve (see \cite[Proposition 4.4]{FPR17a});
 the Hodge type is determined by the curve only, because the del Pezzo surface $\bar X$ has $p_g(\bar X ) = 0 $.

 Thus we can degenerate further by choosing $B_0$ and $D$ such that they are tangent in one or two points. Each tangency will introduce a node in $\bar D$ and by \eqref{eq: MHS normalisation curve} introduce a one-dimensional piece of weight zero in the Hodge structure.
 This gives the degeneration of Hodge types $\HT{0,2}\to \HT{1,1} \to \HT{2,0}$.

To obtain the entire string $\HT{0,0}\to\HT{0,1}\to \HT{0,2}\to \HT{1,1} \to \HT{2,0}$, we consider  a  degeneration  of the branch divisor $B$.
 Take  $B = B_0 +D_1+D_2$, with $D_1$ and $D_2$ linearly equivalent but distinct. For $D_1$ and $D_2$  general we have negligible singularities, i.e.  Hodge type  $\HT{0,0}$.
 Letting $D_1$ and $D_2$ pass through a point $p_0\in B_0$, with suitable tangent conditions, we obtain a $[3,3]$-point, hence a degeneration of  type  $\HT{0,1}$.
 Letting $D_1=D_2$, we get our surface $X$.

  \item[$(\deg d, \deg g) = (4,2)$]  In this case $\bar X =\IP^2$ and  \eqref{eq: MHS normalisation} gives an exact sequence \[0 \to H^1(D, \IC) \to H^1(\bar D, \IC) \to H^2 (X, \IC) \to H^2(\IP^2, \IC) \to 0.\]
 where $D$ is an elliptic curve and $\bar D$ is a genus 3 curve  (see \cite[Proposition 4.4]{FPR17a});
 as in the previous case we can choose $D$ and $B_0$ to be smooth but  tangent in one or two points  
  to achieve the
 degeneration of Hodge types $\HT{0,2}\to \HT{1,1} \to \HT{2,0}$.

 Note how in the case of two tangencies the Hodge structure on $H^1(\bar D, \IC)$ still has a part of weight one, which comes from  the elliptic curve and
  thus does not show up in the Hodge structure on $H^2(X, \IZ)$.

 As above, to obtain the entire string of Hodge type degenerations, we consider
 $B = B_0 +D_1+D_2$, with $D_1$ and $D_2$ linearly equivalent  passing through a point $p_0\in B_0$,  so that we obtain a $[3,3]$-point, and then
 letting $D_1=D_2$, we get our surface $X$.

  \item[$(\deg d, \deg g) = (1,8)$]
 $\bar X$ is a   singular K3 surface (it has two singularities of type $A_1$ over the vertex of the cone), $D$ is rational curve and $\bar D$ is a genus one curve.
 In this case it is not enough to make $B_0$ and $D$ tangent, as we also need to degenerate the K3 surface to get rid of  the $(2,0)$ part of its Hodge structure.

 This can be achieved by first letting $B_0$ acquire an ordinary quadruple point and then a degenerate quadruple point where two local branches are tangent.
Identifying  $\IP(1,1,2)$ with the quadric cone $\kq_2 \subset \IP^3$, we can obtain such degeneration  choosing four suitable plane sections through a point.

 This gives the degeneration $\HT{0,1}\to \HT{0,2} \to \HT{1,1}$ if $D$ remains general.
If in addition we allow $D$ to be tangent to $B_0$ then $\bar D$ will be a nodal curve of genus one, giving the missing degenerations from Figure \ref{fig:hodge-diagram-diagram}.

   \item[$(\deg d, \deg g) = (3,4)$]
      In this case $\bar X$ is a singular del Pezzo surface of degree four, $D$ is rational curve and $\bar D$ is a genus two curve.
      Again the Hodge type is carried by $\bar D$ and we can degenerate  $\HT{0,2}\to \HT{1,1} \to \HT{2,0}$  by letting $D$ and $B_0$ acquire one or two tangencies.

   As in the previous cases  we obtain the entire string of Hodge type degenerations,  considering
 $B = B_0 +D_1+D_2$, with $D_1$ and $D_2$ linearly equivalent  passing through a point $p_0\in B_0$, and then
 letting $D_1=D_2$.
    \item[$(\deg d, \deg g) = (5,0)$]
    This is the first case we encounter where $X$ is reducible, with $\bar D = D\sqcup D$ and \eqref{eq: MHS normalisation} becomes
     \begin{multline*}
0 \to H^1(D, \IC) \to H^1(D, \IC)^{\oplus2} \to H^2 (X, \IC) \to\\ \to  H^2(\kq_2, \IC)^{\oplus 2}\oplus H^2(D, \IC) \to H^2( D, \IC)^{\oplus 2} \to 0.
     \end{multline*}
    We see that the Hodge type is purely determined by $D$ and can produce the degeneration
    $\HT{0,2}\to \HT{1,1} \to \HT{2,0}$  by letting $D$ acquire one or two nodes.
 \end{description}

 \subsection{Two K3 surfaces}
It was observed in \cite[Example 4.7]{FPR17a} that one can obtain examples of non-normal I-surfaces as follows:

Consider in $\IP^3$ two planes $P_1$ and $P_2$ and let $B = B_1 + B_2$ be a quintic section of $P_1\cup P_2$. Let $L$ be the intersection line and $L_i\subset P_i$ its preimages. Then the double cover
\[ \begin{tikzcd}
    \bar X = X_1 \sqcup X_2 \rar{\pi} \dar[swap]{\text{branched over $B_i + L_i$}} & X \dar\\ P_1 \sqcup P_2 \rar & P_1\cup P_2\\
   \end{tikzcd}
\]
defines, for a general  choice of quintic section a 2-Gorenstein I-surface.

For a general quintic sections the curves $B_i$ are smooth and intersect the lines $L_i$ transversely, so $X_i$ is a singular K3 surface with five nodes over $L_i \cap B_i$. It is at these points that $X$ is not Gorenstein.

To compute the mixed Hodge structure we use \eqref{eq: MHS normalisation} again and get
     \[0  =  ( H^1(L_1, \IC) \oplus H^1(L_2, \IC) )\to H^2 (X, \IC) \to H^2(X_1, \IC)\oplus H^2(X_2, \IC)  \oplus H^2( L, \IC)  \]
so that $X$ has Hodge type $\HT{0,0}$ if the branch curve is general.
\footnote{One might want to compare it  with the fact that if we take two elliptic curves meeting in a node, then the Hodge structure on the first cohomology is still pure, despite the curve being singular.}

We can get all possible degenerations from Figure \ref{fig:hodge-diagram-diagram} by letting  $B_1$ and then $B_2$ independently acquire first an ordinary and then a degenerate quadruple point.

Explicitly, pick a point $p_i \in P_i$ and consider three general planes through both points. The resulting ordinary triple point does not affect the Hodge type, since it gives rise to a canonical singularity.
Now it is easy to choose  different quadrics, which have the required intersection or tangency at the $p_i$.

 \subsection{Double covers of a different cone}\label{4B}

 As proved in \cite{FPRR22} an I-surface with a unique T-singularity of type $\frac{1}{18}(1,5)$ is a complete intersection in $\IP(1,1,2,3,5)$ 
 given by equations
 \[ x_1^3 - yx_2=z^2 - f_{10}(x_1, x_2, y, u)=0\]
  (where $\deg x_i=1, \deg y=2, \deg u=3, \deg z=5$)
for a sufficiently general polynomial $f_{10}$ and we can thus consider $X$ as a double cover over a rational surface
\[ X \to S = \Proj \IC[x_1, x_2, y, u]/(x_1^3 -x_2 y)\]
branched over $B = \{ f_{10} = 0\}$.

For a general branch divisor $B$ the surface $X$ has a unique singularity, which is rational, and thus has Hodge type $\HT{0,0}$.
 We want to explore some degenerations to non-normal surfaces  by degenerating the branch locus but have to be more careful than in Section \ref{sect: cone 1} to remain slc.


Thus for simplicity we will only consider the case where the non reduced part  of   the branch divisor avoids the point $(0:0:0:1) \in \IP(1,1,2,3)$, that is, we consider equations of the form
 \[ f_{10} = g_4(x_1, x_2, y, u)\cdot( u - h_3(x_1, x_2, y))^2.\]
%
Then the homogenous coordinate ring of the non-reduced part of the branch divisor  $D$ is
\[ \IC[x_1, x_2, y, u] / (x_1^3-  x_2y , u - h_3) \isom \IC[x_1, x_2, y ] / (x_1^3 - x_2 y ),\]
which describes a smooth rational curve and  does not depend on $h_3$ at all.

The normalisation $\bar X$ is given by adjoining the ratio $w:=z/(u-h_3)$ of weight $2$ to the coordinate ring of $X$. Thus we may eliminate $z$ using $w$ and $\bar X$ is
\[ \{ x_1^3 -x_2 y = w^2 - g_4 = 0\} \subset \IP(1,1,2,3,2).\]
This is a rational surface with ample anti-canonical sheaf and for general $g_4$, $\bar X$ has two $A_1$ singularities occurring at the intersection points of $\bar X$ and $\IP(2,2)\subset \IP(1,1,2,3,2)$  and one $\frac1{18}(1,5)$ singularity, occurring at  $(0:0:0:1:0)$.

In total, by \eqref{eq: MHS normalisation} the Hodge type is completely controlled by $\bar D$, which is the induced double cover of $D$ described by using $h_3$ to eliminate $u$ from the equations for $\bar X$:
\[ \{ x_1^3 -x_2 y = w^2 - g_4( x_1, x_2, y, h_3) = 0\}\subset \IP(1,1,2,2)\]

For suitable choices of $g_4$ and $h_3$ we can arrange for $\bar D$ to be either smooth of genus 2   or to have  one or two nodes. 
These give rise to degenerate cusps on $X$ thus giving the degenerations $\HT{0,2}\to \HT{1,1} \to \HT{2,0}$.


\subsubsection{A K3 surface and a rational surface}

We consider now a non-normal degeneration of Example \ref{4B}. Let $X$ be the complete intersection in $\IP(1,1,2,3,5)$ given by equations
\[x_2y=z^2-f_{10}(x_1,x_2,y,u)=0\]
for sufficiently general $f_{10}$. As in section \ref{4B}, $X$ is a double cover of $S\colon(x_2y=0)$ in $\IP(1,1,2,3)$.
This time, $S$ has two components $S_1\colon(y=0)$ and $S_2\colon(x_2=0)$ which are the weighted projective planes $S_1\cong\IP(1,1,3)$ and $S_2\cong\IP(1,2,3)$.
They are glued along a common weighted projective line $L:=\IP(1,3)\cong\IP^1$. The double cover $X_1$ of $S_1$ is a singular del Pezzo surface with one $\frac13(1,1)$ singularity
 over the point $(0:0:1)$ in $\IP(1,1,3)$ (see \cite{suzuki-reid} for many more details). The double cover $X_2$ of $S_2$ is a K3 surface with an $A_2$ singularity over the point $(0:0:1)$ in $\IP(1,2,3)$. 
 The double curve of $X$ is an elliptic curve $D$ which is the preimage of $L$. The two singular points are identified to give a singularity on $X$ which is analytically isomorphic to the quotient of $\{x_2y=0\}\subset\IA^3$ by the order three action
  $(x_2,y,z)\mapsto(\omega x_2,\omega^2 y,\omega^2 z)$ where $\omega$ is a primitive third root of unity.

Using \eqref{eq: MHS normalisation} again, we compute the mixed Hodge structure:
\[0  \to  H^1(D, \IC) \to H^1(D, \IC)^{\oplus2} \to H^2 (X, \IC) \to H^2(X_1, \IC)\oplus H^2(X_2, \IC)\oplus H^2(D, \IC),\]
so $X$ has Hodge type $\HT{0,1}$.

If we allow $D$ to get a node (i.e. $f_{10} $  to become tangent to $L$)
 then we get Hodge type  $\HT{1,0}$ and so we have recovered  all the types in Figure  \ref{fig:hodge-diagram-diagram}.

\newcommand{\etalchar}[1]{$^{#1}$}
\def\cprime{$'$}

 \end{document}